\documentclass[11pt,reqno]{amsart}
\usepackage{fullpage}
\usepackage[T1]{fontenc}                                   
\usepackage{amsfonts}
\usepackage[utf8]{inputenc}                               
\usepackage{comment}                                       
\usepackage{mparhack}                                      
\usepackage{amsmath,amssymb,amsthm,mathrsfs,eucal}                      
\usepackage{booktabs}                                                                          
\usepackage{graphicx,subfig}                                                                
\usepackage{wrapfig}                                                                            
\usepackage[bookmarks=true,colorlinks=true]{hyperref}                      
\usepackage{bm}                                                                                   

\newtheorem{defn}{Definition}[section]
\newtheorem{thm}{Theorem}[section]
\newtheorem{prop}{Proposition}[section]
\newtheorem{lem}{Lemma}[section]
\newtheorem{cor}{Corollary}[section]

\newtheorem{rem}{Remark}[section]

\DeclareMathOperator*{\argmin}{argmin}

\newcommand{\R}{\mathbb{R}}
\newcommand{\Rd}{{\mathbb{R}^{d}}}
\newcommand{\Rdd}{{\mathbb{R}^{2d}}}

\newcommand{\mpd}{{\mathscr{P}_{2}(\Rd)}}
\newcommand{\mpdd}{{\mathscr{P}_{2}(\Rdd)}}
\newcommand{\mpda}{{\mathscr{P}_{2}^a(\Rd)}}

\newcommand{\mP}{\mathscr{P}(\Rd)}

\newcommand{\mF}{{\mathcal{F}}}
\newcommand{\mE}{{\mathcal{E}}}

\newcommand{\ee}{\varepsilon}

\newcommand{\Rn}{\mathbb{R}^n}

\newcommand{\rhou}{\rho_{1}}
\newcommand{\rhod}{\rho_{2}}
\newcommand{\etau}{\eta_1}
\newcommand{\etad}{\eta_2}
\newcommand{\ddt}{\frac{d}{dt}}
\newcommand{\mmu}{\bm{\mu}}
\newcommand{\rrho}{\bm{\rho}}

\newcommand{\rrhot}{\bm{\rho_{\tau}}}

\newcommand{\rrhotn}{\bm{\rho}_{\tau}^{n}}

\newcommand{\rrhotnn}{\bm{\rho}_{\tau}^{n+1}}

\newcommand{\nnu}{\bm{\nu}}
\newcommand{\rhout}{\rho_{1,\tau}}

\newcommand{\rhodt}{\rho_{2,\tau}}

\newcommand{\rhoit}{\rho_{i,\tau}}

\newcommand{\rhoitn}{\rho_{i,\tau}^{n}}

\newcommand{\rhoitnn}{\rho_{i,\tau}^{n+1}}
\newcommand{\etaitnn}{\eta_{i,\tau}^{n+1}}
\newcommand{\rhojtn}{\rho_{j,\tau}^{n}}

\newcommand{\rhoutn}{\rho_{1,\tau}^{n}}

\newcommand{\rhoutnn}{\rho_{1,\tau}^{n+1}}
\newcommand{\etautnn}{\eta_{1,\tau}^{n+1}}
\newcommand{\rhodtn}{\rho_{2,\tau}^{n}}

\newcommand{\rhodtnn}{\rho_{2,\tau}^{n+1}}
\newcommand{\etadtnn}{\eta_{2,\tau}^{n+1}}

\newcommand{\mW}{\mathcal{W}}
\newcommand{\mK}{\mathcal{K}}
\newcommand{\su}{\bm{S}_{\mE_1}}
\newcommand{\sd}{\bm{S}_{\mE_2}}
\newcommand{\si}{\bm{S}_{\mE_i}}

\newcommand{\id}{\mathrm{id}}

\begin{document}
\author{M. Di Francesco \and A. Esposito \and S. Fagioli}
\address{M. Di Francesco, A. Esposito, Simone Fagioli - DISIM - Department of Information Engineering, Computer Science and Mathematics, University of L'Aquila, Via Vetoio 1 (Coppito)
67100 L'Aquila (AQ) - Italy}
\email{marco.difrancesco@univaq.it}
\email{antonio.esposito3@graduate.univaq.it}
\email{simone.fagioli@univaq.it}


\title[Nonlinear cross-diffusion-interaction systems]{Nonlinear degenerate cross-diffusion systems with nonlocal interaction}
\date{}

\begin{abstract}
\noindent
We investigate a class of systems of partial differential equations with nonlinear cross-diffusion and nonlocal interactions, which are of interest in several contexts in social sciences, finance, biology, and real world applications. Assuming a uniform "coerciveness" assumption on the diffusion part, which allows to consider a large class of systems with degenerate cross-diffusion (i.e. of porous medium type) and relaxes sets of assumptions previously considered in the literature, we prove global-in-time existence of weak solutions by means of a semi-implicit version of the Jordan-Kinderlehrer-Otto scheme. Our approach allows to consider nonlocal interaction terms not necessarily yielding a formal gradient flow structure.
\end{abstract}

\maketitle
\section{Introduction}
This paper is devoted to the mathematical theory of the system of PDEs with nonlinear cross-diffusion and nonlocal interactions
\begin{equation}\label{sist}
\begin{cases}
	                 \partial_{t}\rhou=\mbox{div}\left(\rhou\nabla A_{\rhou}(\rhou,\rhod)+\rhou\nabla H_{1}\ast\rhou+\rhou\nabla K_{1}\ast\rhod\right) \\
		        \partial_{t}\rhod=\mbox{div}\left(\rhod\nabla A_{\rhod}(\rhou,\rhod)+\rhod\nabla H_{2}\ast\rhod+\rhod\nabla K_{2}\ast\rhou\right).
\end{cases}
\end{equation}
Here, $\rhou$ and $\rhod$ are two nonnegative functions defined on $\Rd\times [0,+\infty)$ modelling the densities of \emph{two population species}. $A=A(\rhou,\rhod)$ is a nonlinear function depending on both species. The functions $H_{1}, H_{2}$ are called \textit{self-interaction} potentials, in that they model interactions between individuals of the same species, whereas $K_{1}, K_{2}$ are called \textit{cross-interaction} potentials since they account for interactions between individuals of opposite species. 

Models of the form \eqref{sist} arise in many contexts in biology, socio-economical sciences, and real world applications. The structure in \eqref{sist} is a natural generalization of the one-species gradient flow
\[\partial_t \rho = \mathrm{div}(\rho\nabla (f(\rho) +W\ast \rho)),\]
which arises in chemotaxis of cells (see e.g. \cite{keller_segel,jager,blanchet}), in the description of animal swarming phenomena (\cite{mogilner,topaz,okubo,capasso}), in the physics of granular media \cite{toscani_granular}, in the modelling of pedestrian movements \cite{helbing}, in geology \cite{sedimentation}, in the modelling of opinion formation \cite{sznajd}, and in the modelling of dislocations for crystal defects \cite{peletier}. 

The many-species version \eqref{sist} is motivated in several contexts. It was derived in the study of bacterial chemotaxis models in which two species compete for one single nutrient, see \cite{conca,espejo}, see also \cite{preziosi} for applications of such approach to the study of tumor growth. In the modelling of pedestrian movements, the use of a two species model was adopted in \cite{degond} to study multi-lane pattern formations. Other nonlocal transport models for pedestrian movements with many species have been studied in \cite{colombo,crippa}. A very interesting framework in which the many species modelling approach \eqref{sist} is gaining increasing attention in the applied mathematical community is that of opinion formation. A first model considering opinion leaders and followers was introduced in \cite{during}. Later on, improved models were considered in \cite{escudero,assertiveness}.
In population biology, the model derived by Bruna and Chapman in \cite{bruna_chapman} can be considered as a paradigm for other existing models in this context. We also mention here the application of this approach to the modelling of predator-prey dynamics, see e.g. \cite{kolokolnikov,fagioli}. In quantitative finance, a model of the form \eqref{sist} can be recovered in the study of socio-economic dynamics of speculative markets with two species of agents (fundamentalists and chartists), see e.g. the book \cite{pareschi}.

Let us now look at the mathematical structure of system \eqref{sist} a bit more closely. Even at a formal structure level, in the \emph{symmetric cross-interaction case} $K_1=K_2=K$ the system \eqref{sist} displays a \emph{gradient flow structure}. More precisely, let
\begin{align*}
&\mF(\rhou,\rhod)=\int_{\R^d}A(\rhou,\rhod) dx + \frac{1}{2}\int_{\R^d}\rhou H_1\ast \rhou dx\\
&\qquad + \frac{1}{2}\int_{\R^d}\rhod H_2\ast \rhod dx + \int_{\R^d}\rhou K \ast \rhod dx.
\end{align*}
Then \eqref{sist} can be formally seen as 
\begin{equation}\label{sistGF}
\begin{cases}
\displaystyle{\partial_{t}\rhou=\mbox{div}\left(\rhou\nabla\frac{\delta \mF}{\delta\rhou} \right)} \vspace{2mm}\\
		       \displaystyle{ \partial_{t}\rhod=\mbox{div}\left(\rhod\nabla \frac{\delta \mF}{\delta\rhod}\right),}
\end{cases}
\end{equation}
where the symbols $\frac{\delta \mF}{\delta\rhou}$ and $\frac{\delta \mF}{\delta\rhod}$ denote "functional" derivatives. Such a gradient flow structure is no longer featured in case $K_1\neq K_2$, since in this case the cross interaction terms in \eqref{sist} cannot be written with the functional derivative formalism used for the other terms (actually, if $K_1=\alpha K_2$ with $\alpha>0$ the gradient flow structure can still be achieved by slightly modifying the metric). In order to stick as much as possible to a gradient flow mechanism, one can use some sort of time-splitting argument in which the cross-interaction potentials $K_{1}\ast\rhod$ and $K_{2}\ast \rhou$ are taken explicit in time. More precisely, set
\begin{align}
&\mF((\rhou,\rhod)\,|\,(\bar\rhou,\bar\rhod))=\int_{\R^d}A(\rhou,\rhod) dx + \frac{1}{2}\int_{\R^d}\rhou H_1\ast \rhou dx\nonumber\\
&\qquad + \frac{1}{2}\int_{\R^d}\rhod H_2\ast \rhod dx + \int_{\R^d}\rhou K_1 \ast \bar\rhod dx+\int_{\R^d}\rhod K_2 \ast \bar\rhou dx.\label{functional_relative_energy}
\end{align}
The object defined in \eqref{functional_relative_energy} is a \emph{relative energy} functional, since it computes the energy at the state $(\rhou,\rhod)$ "with respect to" the state $(\bar\rhou,\bar\rhod)$, where the latter only affects the cross interaction terms. Then, for a fixed time step $\Delta t>0$, denoting $\rhou^n=\rhou(n\,\Delta t)$, $\rhod^n=\rhod(n\,\Delta t)$, the implicit-explicit time-discretized system
\begin{equation}\label{sistGF2}
\begin{cases}
	               \displaystyle{ \frac{\rhou^{n+1}-\rhou^n}{\Delta t}
                    =\mbox{div}\left(\rhou^{n+1}\nabla\frac{\delta \mF((\rhou^{n+1},\rhod^{n+1})\,|\,(\rhou^n,\rhod^n))}{\delta\rhou^{n+1}} \right)}\\
		       \displaystyle{\frac{\rhod^{n+1}-\rhod^n}{\Delta t}=\mbox{div}\left(\rhod^{n+1}\nabla\frac{\delta \mF((\rhou^{n+1},\rhod^{n+1})\,|\,(\rhou^n,\rhod^n))}{\delta\rhod^{n+1}} \right),}
\end{cases}
\end{equation}
is a potential candidate to approximate our system \eqref{sist} as $\Delta t\searrow 0$. Indeed, it is very easy to check that the functional derivatives with respect to $\rhou^{n+1}$ and $\rhod^{n+1}$ in \eqref{sistGF2} give rise to an approximated version of the right hand side of \eqref{sist}. We will come back to this aspect later on in this paper. 

Putting aside the gradient flow structure of \eqref{sist} for a moment, we observe that the system \eqref{sist} consists of second order terms, or diffusion terms, and of nonlocal interaction terms. The latter do not constitute a major problem for the existence theory in case the potentials $K_1, K_2, H_1, H_2$ involved are smooth enough, which in fact will be the case in our paper. The diffusion part is actually the most challenging part of \eqref{sist}, so let us single it out and try to understand its structure in detail. Expanding the terms involving $A(\rhou,\rhod)$ in \eqref{sist} (assume $A$ is smooth enough), we obtain
\begin{equation}\label{sist_expanded}
\begin{cases}
	                 \partial_{t}\rhou=\mbox{div}\left(\rhou A_{\rhou,\rhou}\nabla \rhou +\rhou A_{\rhou,\rhod}\nabla\rhod +\rhou\nabla H_{1}\ast\rhou+\rhou\nabla K_{1}\ast\rhod\right) \\
		        \partial_{t}\rhod=\mbox{div}\left(\rhod A_{\rhod,\rhod}\nabla \rhod +\rhod A_{\rhod,\rhou}\nabla\rhou+\rhod\nabla H_{2}\ast\rhod+\rhod\nabla K_{2}\ast\rhou\right).
\end{cases}
\end{equation}
Hence, system \eqref{sist_expanded} can be potentially seen as a nonlinear parabolic system with a nonlocal perturbation. In fact, \eqref{sist_expanded} features a \emph{diffusion matrix}
\[D(\rhou,\rhod)=
\left[
\begin{matrix}
\rhou A_{\rhou,\rhou}(\rhou,\rhod) & \rhou A_{\rhou,\rhod}(\rhou,\rhod) \\
\rhod A_{\rhod,\rhou}(\rhou,\rhod) & \rhod A_{\rhod,\rhod}(\rhou,\rhod)
\end{matrix}\right].
\]
A well understood, by now standard theory developed many decades ago in \cite{ladyzhenskaya} shows that a global well-posedness result for such a system could be proven if the diffusion matrix $D(\rhou,\rhod)$ would be symmetric and positive definite, where $D$ has to be replaced by its symmetric part $(D+D^T)/2$ in case $D$ is not symmetric. However, a simple calculation shows that 
\[\mathrm{det}(D+D^T)=4\rhou\rhod\mathrm{det} D^2 A-A_{\rhou\rhod}^2(\rhou-\rhod)^2,\]
which shows that the determinant of $(D+D^T)/2$ may be negative in case $A_{\rhou\rhod}\neq 0$ and $|\rhou-\rhod|\gg 1$, even in cases in which the Hessian of $A$ is positive definite. Roughly speaking, the condition $A_{\rhou\rhod}\neq 0$ means that \emph{cross-diffusion} may be present, in that each species has a flux component following the gradient of the other species. As it is well-known for instance in the theory of chemotaxis modelling, see e.g. \cite{blanchet}, and in other examples studied previously e.g. \cite{stara}, there are cases in which solutions may break down in a finite time if the strict, uniform parabolicity condition $D(\rhou,\rhod)\geq c\mathbb{I}$ ($c>0$) is not satisfied. The main issue with the lack of global existence is related with the need of a priori estimates on the solutions $(\rhou,\rhod)$ and its space derivatives. This problem was deeply investigate by Amann in the nineties, see e.g. \cite{amann}, and by other authors, see \cite{pierre,nguyen}. Alternative concepts of parabolicity have been formulated which still allow to prove a global existence result in significant cases. A relevant example is that of Petrowski's parabolicity, see e.g. \cite{kreiss}. With the notation of \eqref{sist_expanded}, such condition requires the existence of a symmetric, positive definite matrix $P$ such that $PD^T+DP$ is (symmetric) positive definite. In the simple example of quadratic diffusion $A(\rhou,\rhod)=\rhou^2+\rhod^2+\rhou\rhod$, the existence of such a positive definite symmetrizer would imply once again a contradiction if e.g. $\rho_1$ is close to vacuum and $\rho_2$ is not (we leave the details to the reader).

The above considerations show that classical parabolic theories fail as long as they do not take into account of the gradient flow structure we outlined above, see e.g. \eqref{sistGF2}. A general existence theory for \eqref{sist} with arbitrary cross-diffusion terms is unlikely to be carried out, except perhaps in a measure solution setting. Moreover, even in cases in which the diffusion term $A(\rhou,\rhod)$ has a nonnegative Hessian, a general existence theory is not provided by classical parabolic theories. It is, by now, well understood that a gradient flow structure whatsoever must be exploited in order to achieve a satisfactory theory. 

With the exception of a recent contribution to this line of research proposed by J\"ungel in \cite{juengel} (see also previous results in \cite{chen_juengel,burger_DF_piet_sch}), in which a formal gradient flow formulation provides the estimates needed to prove global existence, most of the current efforts to a global existence theory for cross-diffusion "gradient systems" are based on the `many species' version of the \emph{Wasserstein gradient flow} theory of \cite{JKO,otto,AGS,CDFFLS}, which essentially allows to make the time-discretization approach of \eqref{sistGF2} rigorously posed in a variational form using the relative energy functional $\mF$ in \eqref{functional_relative_energy}. Such an approach has been already successfully used in \cite{DFF} in the case $A\equiv 0$, i.e. for a system of nonlocal interaction equations with two species and non symmetric cross-interactions. The first result in which this technique was used with a cross-diffusion term is the one proven in \cite{LM} (no interaction terms). Other results
\cite{L,CL1,CL2} only apply to diagonal diffusion $A(\rhou,\rhod)=a_1(\rhou)+a_2(\rhod)$ and in some cases only on bounded domains. A deep result \cite{matthes_zinsl}, proven in one space dimension, introduces a notion of displacement convexity in the many species framework and proves an existence result under a uniform convexity assumption on $A$ and in one space dimension.

\medskip

In our paper, we focus in particular on the case of \emph{degenerate diffusion}, namely functions $A(\rho_1,\rho_2)$ that behave like power laws $\rho_i^{m_i}$ with $m_i>1$, a situation which potentially gives rise to possible loss of regularity near the vacuum state, see \cite{vazquez}. More in detail, we propose the following improvements to the theory:
\begin{enumerate}
\item We provide, for the first time, a theory which combines cross-diffusion effects with nonlocal interaction terms, and the latter need not feature symmetric cross-interactions (i.e. $K_1$ and $K_2$ are linearly independent).
\item In previous results (see \cite{matthes_zinsl}), the nonlinear diffusion function $A(\rho_1,\rho_2)$ is required to be uniformly convex, whereas our set of assumptions allows for some degeneracy in the convexity.
\item Our result holds in arbitrary dimension $d\geq 1$ and on the whole $\R^d$.
\item The growth conditions on the diffusion function $A$ allow for non-homogeneous dependencies with respect to each species.
\end{enumerate}

\noindent
We now state our assumptions on $A:\R^2\to\R$:
\begin{itemize}
\item[(D1)] There exist $m_1,m_2>1$ and a function $B\in C^2([0,+\infty)^2;\,[0,+\infty))$ such that $B(0,0)=0$, $\nabla B(0,0)=(0,0)$, and
\[A(\rho_1,\rho_2)=B\left(\rho_1^{\frac{m_1}{2}},\rho_2^{\frac{m_2}{2}}\right);\]
\item [(D2)] Given
\begin{equation}\label{eqMi}
M_i:=\frac{m_i\, (d+2)}{d}
\end{equation}
with $i=1,2$ and $m_1, m_2$ as in (D1), we assume there exist $\alpha_i\in [m_i,M_i)$, $i=1,2$, such that the functions
\[\eta_1^2 B_{\eta_1,\eta_1},\quad \eta_2^2 B_{\eta_2,\eta_2},\quad \eta_1 \eta_2 B_{\eta_1,\eta_2}\]
grow at most as 
\[\eta_1^a\eta_2^b,\qquad \hbox{with}\qquad \frac{a\, m_1}{2\alpha_1}+\frac{b\, m_2}{2\alpha_2} =1\]
for large (positive) $\eta_1$ and $\eta_2$, where the constants $a$ and $b$ can vary for each term;
\item [(D3)] The function $B$ in (D1) is uniformly convex, or equivalently 
\[<D^2A(\xi_1,\xi_2) V,V>\ \ge C_1(\xi_1^{m_1-2}||V_1||^2+\xi_2^{m_2-2}||V_2||^2)\] 
for all 
$(\xi_1,\xi_2)\in (0,+\infty)^2$ and $(V_1,V_2)\in\R^2$ and for some constant $C_1>0$. Again $m_1$ and $m_2$ are as in the  assumption (D1).
\end{itemize}
As a trivial consequence of (D1) and (D3), it is easy to check that there exists a constant $C>0$ such that
\begin{equation}\label{exD2}
A(\rho_1,\rho_2)\geq C (\rho_1^{m_1}+\rho_2^{m_2}).
\end{equation}
\begin{rem}
\emph{Before we discuss the assumptions on the nonlocal part, let us produce some significant examples of diffusion functions $A(\rho_1,\rho_2)$ satisfying out set of assumptions (D1)-(D3). 
\begin{itemize}
\item [(1)] Every positive definite quadratic form with respect to powers $\rho_1^{m_1/2}$ and $\rho_2^{m_2/2}$ for arbitrary $m_1,m_2>1$ is allowed in our set of assumptions. In particular, we can consider diffusion terms of the form
\[A(\rho_1,\rho_2)=a\rho_1^{m_1} + b\left(\rho_1^{\frac{m_1}{2}}+\rho_2^{\frac{m_2}{2}}\right)^2,\qquad a,b>0.\]
\item [(2)] Any diffusion function of the form
\[A(\rho_1,\rho_2)=a\rho_1^{m_1}+b\rho_2^{m_2}+p(\rho_1+\rho_2),\qquad a,b>0,\]
with $m_1,m_2>1$ and $p$ a nonnegative, (not necessarily uniformly) smooth, convex function satisfies our set of assumptions as long as $p$ grows slower than $\rho_1^{m_1\left(1+\frac{2}{d}\right)} + \rho_2^{m_2\left(1+\frac{2}{d}\right)}$ as $|(\rho_1,\rho_2)|\rightarrow +\infty$.
\item [(3)] An interesting diffusion function arising in population dynamics (see e.g. \cite{bertsch,sorting,zoology}) is $A=(\rho_1+\rho_2)^2$. Clearly (see the discussion at the end of this section), this function does not satisfy our set of assumptions. On the other hand, if we perturb it by a power of just one species with an exponent less than $2$, namely if we consider
\[A(\rho_1,\rho_2)=a\rho_1^{m}+b(\rho_1+\rho_2)^2,\qquad a,b>0,\]
then the above assumptions are satisfied on domains of the form $0\leq \rho_1\leq R$ for an arbitrary constant $R>0$.
\end{itemize}
}
\end{rem}

\noindent
All the interaction potentials $H_1, H_2, K_1, K_2$ are requested to satisfy 
\[H_i(-x)=H_i(x)\,,\qquad i=1,2,\]
plus the following assumptions ($H$ denotes any of the above four potentials in the next two conditions):
\begin{itemize}
\item [(HK1)] $H\in C(\Rd)\cap C^1(\R^d\setminus\{0\})$, $H(0)=0$. Moreover, there exist $C_1,C_2>0$ and $0<\alpha<2$ such that $-C_1 (1+|x|^{\alpha}) \leq H(x) \leq C_2(1+|x|^{2})$ for all $x\in \R^d$;
\item [(HK2)] $\Delta H$ is a Radon measure such that $\Delta H\le \bar{C}$ in $\mathfrak{D'}(\Rd)$ for some constant $\bar{C}$.
\end{itemize}

\noindent
More specific regularity assumptions are required on each potential.
The self-interaction potentials $H_1$ and $H_2$ satisfy ($i=1,2$):
\begin{itemize}
\item [(H1)] $\nabla H_i\in L_{loc}^{\infty}(\Rd)$ and $|\nabla H_i(x)|\le C(1+|x|),\ \forall x\in\Rd$.
\end{itemize}

\noindent
For the cross-interaction potentials $K_1$ and $K_2$, we additionally require (for $i=1,2$):
\begin{itemize}
\item [(K1)] $K_i$ and $\nabla K_i$ are globally Lipschitz on $\Rd$.
\end{itemize}
We notice that the symmetry of the potential is only required for the self-interaction potentials, not for the cross-interaction ones.

\smallskip
\noindent
While the assumptions on the interaction potentials are mainly involving some minimal regularity needed to perform suitable estimates, our requirements on the diffusion part amount essentially to assuming that the Hessian of $A$ has a dominant diagonal part with respect to the mixed terms. This models the situation in which the cross-diffusion part is weaker compared to the diagonal diffusion terms, which allows to detect better estimates on each single component. 

Our set of assumptions on the diffusion part is somewhat sharp in case one aims at achieving a diffusion-induced regularity provided by the dissipation of the functional $\mF$. This fact can be easily seen in the recent \cite{sorting}, in which discontinuous steady states are generated with $A(\rhou,\rhod)=(\rhou+\rhod)^2$, see similar results in \cite{zoology}. In those cases, a $BV$ estimate strategy must be developed, but there are no results allowing to do so at the moment for systems like \eqref{sist}. It is worth mentioning that the emergence of discontinuous "segregated" states as those studied in \cite{sorting} was previously detected in other situations in which nonlinear cross-diffusion couples with lower order effects, such as reaction terms, see e.g. \cite{bertsch}. 

The paper is structured as follows. In Section \ref{pre} we recall some preliminaries on optimal transport theory needed to perform the Jordan-Kinderlehrer-Otto scheme. In Section \ref{eulereq} we state and prove our existence result, the precise statement of which is provided in Theorem \ref{mainthm}.

\section{Preliminaries on optimal transport theory}\label{pre}
In the following, for a given integer $d\in\mathbb{N}$, we shall denote with $\mP$ the space of all probability measures on $\Rd$ and with $\mpd$ the set of all probability measures with finite second moment, i.e.
$$
\mpd=\left\{\rho\in\mP:m_2(\rho)<+\infty\right\},
$$
where
$$m_2(\rho)=\int_{\Rd}|x|^2\,d\rho(x).$$
Consider now a measure $\rho\in\mP$ and a Borel map $T:\Rd\to\Rn$. We denote by $T_{\#}\rho$ the push-forward of $\rho$ through $T$, defined by
$$
\int_{\Rn}f(y)\,dT_{\#}\rho(y)=\int_{\Rd}f(T(x))\,d\rho(x) \qquad \mbox{for all $f$ Borel functions on}\ \Rn.
$$
Let us recall the $2$-Wasserstein distance between $\mu_1,\mu_2\in \mpd$ defined by
\begin{equation}\label{wass}
W_2^2(\mu_1,\mu_2)=\min_{\gamma\in\Gamma(\mu_1,\mu_2)}\left\{\int_{\Rdd}|x-y|^2\,d\gamma(x,y)\right\},
\end{equation}
where $\Gamma(\mu_1,\mu_2)$ is the class of all transport plans between $\mu_1$ and $\mu_2$, that is the class of measures $\gamma\in\mpdd$ such that, denoting by $\pi_i$ the projection operator on the $i$-th component of the product space, the marginality condition
$$
(\pi_i)_{\#}\gamma=\mu_i \quad \mbox{for}\ i=1,2
$$
is satisfied. Setting $\Gamma_0(\mu_1,\mu_2)$ as the class of optimal plans, i.e. minimizers of \eqref{wass}, we can write the Wasserstein distance as
$$
W_2^2(\mu_1,\mu_2)=\int_{\Rdd}|x-y|^2\,d\gamma(x,y), \qquad \gamma\in\Gamma_0(\mu_1,\mu_2).
$$
\begin{rem}\label{momp}
\emph{From the definition of $W_2$ and from the inequality $|y|^2\le2|x|^2+2|x-y|^2$ it is possible to deduce that
$$
m_2(\rhou)\le2m_2(\rho_0)+2W_2^2(\rho_0,\rhou), \qquad \forall \rho_0,\rhou\in\mpd.
$$}
\end{rem}\noindent
The space $(\mpd,W_2)$ is a complete metric space and it can be seen as a length space (see for instance \cite{AGS},\cite{V1},\cite{V2}). For our purpose it is important to recall the definition of $\lambda$-geodesically convex functionals in $\mpd$ and other related concepts. First of all, a curve $\mu:[0,1]\to\mpd$ is a constant speed geodesic if $W_2(\mu(s),\mu(t))=(t-s)W_2(\mu(0),\mu(1))$ for $0\le s\le t\le1$; due \cite[Theorem 7.2.2]{AGS}, a constant speed geodesic connecting $\mu_1$ and $\mu_2$ can be written as
$$
\gamma_t=((1-t)\pi_1+t\pi_2)_{\#}\gamma,
$$
where $\gamma\in\Gamma_0(\mu_1,\mu_2)$. A functional $\phi:\mpd\to(-\infty,+\infty]$ is $\lambda$-geodesically convex in $\mpd$ if for every pair $\mu_1,\mu_2\in\mpd$ there exists $\gamma\in\Gamma_0(\mu_1,\mu_2)$ and $\lambda\in\R$ such that
$$
\phi(\gamma_t)\le(1-t)\phi(\mu_1)+t\phi(\mu_2)-\frac{\lambda}{2}t(1-t)W_2^2(\mu_1,\mu_2) \qquad \forall t\in[0,1].
$$
Now, let us mention the concept of \textit{k-flow}, which is linked with the $\lambda$-convexity along geodesics.
\begin{defn}[k-flow]
A semigroup $S_{H}:[0,+\infty]\times\mpd\to\mpd$ is a $k$-flow for a functional $H:\mpd\to\R\cup\{+\infty\}$ with respect to $W_2$ if, for an arbitrary $\rho\in\mpd$, the curve $t\mapsto S_H^t\rho$ is absolutely continuous on $[0,+\infty[$ and satisfies the evolution variational inequality (EVI)
\begin{equation}
\frac{1}{2}\frac{d^+}{dt}W_2^2(S_H^t\rho,\tilde{\rho})+\frac{k}{2}W_2^2(S_H^t\rho,\tilde{\rho})\le H(\tilde{\rho})-H(S_H^t\rho)
\end{equation}
for all $t>0$, with respect to every reference measure $\tilde{\rho}\in\mpd$ such that $H(\tilde{\rho})<\infty$.
\end{defn}
\begin{rem}
\emph{Just for the sake of completeness, we specify that saying $H$ admits a $k$-flow is equivalent to the $\lambda$-convexity of $H$ along geodesics, cf.\cite{AGS,DFM,MMCS} for further details.}
\end{rem}
\noindent
The previous definition will be applied in section \ref{eulereq} to the entropy functional $H=\int_{\Rd}\rho(x)\log\rho(x)\,dx$, which possesses a $0$-flow $S_H$ given by the heat semigroup (cf. \cite{JKO,V1,DS}). Indeed, the curve $t\mapsto \eta(t):=\bm{S}_H^t\eta_0$ solves the heat equation with a given initial datum $\eta_0\in\mpd$ in the classical sense.
\begin{rem}\label{prod}
\emph{Since we are dealing with the evolution of two species, we need to work in the product space $\mpd\times\mpd$ equipped with a product metric. We shall denote the elements of a product space using bold symbols, for instance
$$
\bm{\rho}=(\rhou,\rhod)\in\mpd\times\mpd
$$
or
$$
\bm{x}=(x_1,x_2)\in\Rd\times\Rd.
$$
The Wasserstein distance of order two in the product space is defined as follows
$$
\mW_2^2(\bm{\mu},\bm{\nu})=W_2^2(\mu_1,\nu_1)+W_2^2(\mu_2,\nu_2)
$$
for all $\bm{\mu},\bm{\nu}\in\mpd\times\mpd$.}
\end{rem}
\noindent
At the end of this section we recall the refined version of the Aubin-Lions Lemma due to Rossi and Savar\'{e} \cite[Theorem 2]{RS}.
\begin{thm}\label{aulirs}
Let $X$ be a Banach space. Consider
\begin{itemize}
\item a lower semi-continuous functional $\mF:X\to[0,+\infty]$ with relatively compact sublevels in $X$;
\item a pseudo-distance $g:X\times X\to[0,+\infty]$, i.e. $g$ lower semi-continuous and such that $g(\rho,\eta)=0$ for any $\rho,\eta\in X$ with $\mF(\rho)<\infty$, $\mF(\eta)<\infty$ implies $\rho=\eta$.
\end{itemize}
Let $U$ be a set of measurable functions $u:(0,T)\to X$, with a fixed $T>0$. Assume further that
\begin{equation}\label{hprossav}
\sup_{u\in U}\int_{0}^T\mF(u(t))\,dt<\infty\quad \text{and}\quad \lim_{h\downarrow0}\sup_{u\in U}\int_{0}^{T-h}g(u(t+h),u(t))\,dt=0\,.
\end{equation}
Then $U$ contains an infinite sequence $(u_n)_{n\in\mathbb{N}}$ that converges in measure (with respect to $t\in(0,T)$) to a measurable $\tilde{u}:(0,T)\to X$.
\end{thm}

\section{Existence of solutions}\label{eulereq}
If we do not consider the diffusion in \eqref{sist}, we fall in the case of purely nonlocal interaction systems, already treated in \cite{DFF}. Summarising, if the cross-interaction potentials are proportional ($K_2=\alpha K_1$, for $\alpha>0$), one can use the theory of gradient flows on probability spaces developed in \cite{AGS},\ \cite{AS}
and \cite{CDFFLS} for nonlocal interaction equations with one species. Otherwise, a \emph{semi-implicit} version of the JKO scheme \cite{DFF} can be adopted to construct solutions in way to stick as much as possible to a variational Wasserstein-based structure, as discussed in the introduction, see \eqref{sistGF2}. Such procedure consists in freezing the non symmetric part of the system and treat it as an external potential. This strategy will be performed in this section, with the goal of detecting weak solutions for \eqref{sist}.

\begin{defn}\label{defweaksolution}
A curve $\bm{\rho}=(\rhou(\cdot),\rhod(\cdot)):[0,+\infty)\longrightarrow\mpd^2$ is a weak solution to \eqref{sist} if 
\begin{itemize}
\item[(i)] $\rho_i\in L^{\alpha_i}([0,T]\times \R^d)$ with $\alpha_i\in (1,M_i)$ and $M_i$ defined in \eqref{eqMi}, for $i=1,2$ and for all $T>0$,
\item [(ii)] $\nabla \rho_i^{m_i/2} \in L^2([0,+\infty)\times \R^d)$ for $i=1,2$,
\item[(iii)] for almost every $t\in [0,+\infty)$ and for all $\phi,\varphi\in C_c^{\infty}(\Rd)$, we have
\begin{equation*}
\begin{split}
\ddt\int\phi(x)\,d\rhou(t,x)&=-\int\nabla A_{\rhou}\cdot\nabla\phi(x)\,d\rhou(x)-\int\int\nabla K_1(x-y)\cdot\nabla\phi(x)\,d\rhod(y)d\rhou(x)\\&-\frac{1}{2}\int\int\nabla H_1(x-y)\cdot(\nabla\phi(x)-\nabla\phi(y))\,d\rhou(y)d\rhou(x)
\end{split}
\end{equation*}
\begin{equation*}
\begin{split}
\ddt\int\varphi(x)\,d\rhod(t,x)&=-\int\nabla A_{\rhod}\cdot\nabla\varphi(x)\,d\rhod(x)-\int\int\nabla K_2(x-y)\cdot\nabla\varphi(x)\,d\rhou(y)d\rhod(x)\\&-\frac{1}{2}\int\int\nabla H_2(x-y)\cdot(\nabla\varphi(x)-\nabla\varphi(y))\,d\rhod(y)d\rhod(x)
\end{split}
\end{equation*}
\end{itemize}
\end{defn}
\noindent

We now introduce the \textit{relative energy} functional $\mF$. Let $\bm{\nu}\in\mpd^2$ be a fixed (time independent) measure; for all $\bm{\mu}\in\mpd^2$ we set
\begin{equation*}
\begin{split}
\mF[\mmu|\nnu]&:=\int_{\Rd}A(\mu_1,\mu_2)\,dx + \frac{1}{2}\int_{\Rd}H_1 * \mu_1\,d\mu_1 + \int_{\Rd}K_1 * \nu_2\,d\mu_1\\&+\frac{1}{2}\int_{\Rd}H_2 * \mu_2\,d\mu_2 + \int_{\Rd}K_2 * \nu_1\,d\mu_2.
\end{split}
\end{equation*}
By abuse of notation, by $A(\mu_1,\mu_2)$ we mean $A$ evaluated on the Radon-Nicodym derivatives of $\mu_1$ and $\mu_2$ (respectively) with respect to the Lebesgue measure in case both $\mu_1$ and $\mu_2$ are $L^1$ densities, otherwise $A$ takes the value $+\infty$. With this notation the functional is well-defined since the interaction potentials are continuous and they grow at most quadratically. 

\noindent
We are now ready to state the main result of our paper.
\begin{thm}\label{mainthm}
Assume that (D1)-(D3), (HK1)-(HK2), (H1), and (K1) are satisfied. 
Let $\rrho_0=(\rho_{1,0},\rho_{2,0})\in \mpd^2$ such that
\[\mF[\rrho_0|\rrho_0]<+\infty.\]
Then, there exists a weak solution to \eqref{sist} in the sense of Definition \ref{defweaksolution}.
\end{thm}

\noindent
We will prove the result in Theorem \ref{mainthm} in the current section. In order to improve the readability we split this section into several subsections, each one corresponding to each step of the proof.

\begin{rem}\label{splitfunc}
\emph{
Let us set
$$
\tilde{\mF}[\mmu]:=\int_{\Rd}A(\mu_1,\mu_2)\,dx + \frac{1}{2}\int_{\Rd}H_1 * \mu_1\,d\mu_1 + \frac{1}{2}\int_{\Rd}H_2 * \mu_2\,d\mu_2,
$$
and
$$
\mK[\mmu|\nnu]:=\int_{\Rd}K_1 * \nu_2\,d\mu_1 + \int_{\Rd}K_2 * \nu_1\,d\mu_2.
$$
Then, we can rewrite $\mF$ as the sum of the previous functionals, i.e.
$$
\mF[\mmu|\nnu]=\tilde{\mF}[\mmu]+\mK[\mmu|\nnu].
$$
Such a notation clarifies the expression "semi-implicit" used throughout the paper, as the part $\tilde{\mF}$ is treated implicitly in the JKO scheme as usual, whereas $\mK[\mmu|\nnu]$ contains terms that are treated explicitly in the time-discretization.}
\end{rem}\noindent

Once we have the \textit{relative energy} functional $\mF$, we can develop the semi-implicit JKO scheme in the spirit of \cite{DFF}. Let $\tau>0$ be a fixed time step and let $\rrho_0=(\rho_{0,1},\rho_{0,2})\in\mpd^2$ be a fixed initial datum such that $\mF[\rrho_0|\rrho_0]<+\infty$. We define a sequence $\{\rrhotn\}_{n\in\mathbb{N}}$ recursively: $\rrhot^0=\rrho_0$ and, for a given $\rrhotn\in\mpd^2$ with $n\geq 0$, we choose $\rrhotnn$ as follows:
\begin{equation}\label{jko}
\rrhotnn\in\argmin_{\rrho\in\mpd^2}\left\{\frac{1}{2\tau}\mW_2^2(\rrhotn,\rrho)+\mF[\rrho|\rrhotn]\right\}.
\end{equation}\noindent
The well-posedness of the sequence \eqref{jko} can be obtained using the same argument in \cite[Lemma 2.3 and Proposition 2.5]{CDFFLS} for each component, since the interaction potentials have a ``suitable'' control from below, the diffusive part of the functional is nonnegative and the integrand $A$ is a nonnegative $C^2$ function.

\begin{rem}
\emph{We assume that the interaction potentials are controlled from below by $C_1(1+|x|^\alpha)$ for $0<\alpha<2$ and $C_1<0$. Such assumption implies that, for a given $\nnu\in\mpd^2$, the functional
$$
\rrho\in\mpd^2\longrightarrow\frac{1}{2\tau}\mW_2^2(\nnu,\rrho)+\mF[\rrho|\nnu]
$$
is bounded from below if $\tau$ is small enough. Indeed, one can use the H\"{o}lder and the (weighted) Young inequalities together with Remark \ref{momp} with $\rrho$ and $\nnu$, in order to have
$$
\frac{1}{2\tau}\mW_2^2(\nnu,\rrho)+\mF[\rrho|\nnu]\ge-C m_2(\nnu)-C,
$$
for a suitable constant $C$ possibly depending on $\alpha$. This provides the boundedness from below needed in the well-posedness of the scheme \eqref{jko}.}
\end{rem}

\subsection{Convergence of the scheme}
A crucial issue in our paper is proving that an appropriate time-interpolation of the sequence defined by the scheme \eqref{jko} converges to a weak measure solution of \eqref{sist}. At this point let us consider the following piecewise constant interpolation defined as follows. Let $T>0$ be fixed. Let $N:=\left[\frac{T}{\tau}\right]$. We set
$$
\rhoit(t)=\rhoitn \qquad t\in((n-1)\tau,n\tau],
$$
with $\rrhotn=(\rhoutn,\rhodtn)$ defined in \eqref{jko}. 
\begin{prop}\label{propaa}
There exists an absolutely continuous curve $\tilde{\rrho}: [0,T]\rightarrow\mpd^2$ such that the piecewise constant interpolation $\rrhot$ admits a subsequence $\rrho_{\tau_k}$ narrowly converging to $\tilde{\rrho}$ uniformly in $t\in[0,T]$ as $k\rightarrow +\infty$.
\end{prop}
\begin{proof}
By the definition of the scheme \eqref{jko} we have
\begin{equation}\label{eq:scheme1}
\begin{split}
&\frac{1}{2\tau}\mW_2^2(\rrhotn,\rrhotnn)\le\mF[\rrhotn|\rrhotn]-\mF[\rrhotnn|\rrhotn]\\&=\int_{\Rd}A(\rhoutn,\rhodtn)\,dx-\int_{\Rd}A(\rhoutnn,\rhodtnn)\,dx\quad +\\&+\sum_{i=1}^{2}\frac{1}{2}\left(\int_{\Rd}H_i * \rhoitn\,d\rhoitn-\int_{\Rd}H_i * \rhoitnn\,d\rhoitnn\right)\quad +\\&+\sum_{i\ne j}\left(\int_{\Rd}K_i * \rhojtn\,d\rhoitn-\int_{\Rd}K_i * \rhojtn\,d\rhoitnn\right).
\end{split}
\end{equation}
Reasoning as in \cite{DFF}, we can use the Lipschitz's hypothesis on $K_i$ as follows for $i,j=1,2$, $i\neq j$. First we use the symmetry of $K$ to get
\begin{align*}
 & \left|\int_{\Rd}K_i * \rhojtn\,d\rhoitn-\int_{\Rd}K_i * \rhojtn\,d\rhoitnn \right|\\
 & \ = \left|\iint_{\Rd\times\Rd}K_i(x-y)\,d\rhojtn(y)\,d\rhoitn(x)-\iint_{\Rd\times\Rd} K_i(t-y)\,d\rhojtn(y)\,d\rhoitnn(t)\right|\\
 & \ = \left|\iiint_{\R^d\times\R^d\times\R^d}(K_i(x-y)-K_i(t-y))\,d \gamma_{i,\tau}^n(x,t)\,d \rhojtn(y)\right|,
\end{align*}
where $\gamma^n_{i,\tau}\in \Gamma_o(\rhoitn,\rhoitnn)$ is an optimal transport plan connecting $\rhoitn$ to $\rhoitnn$. Now, due to the Lipschitz condition (K1) we get
\begin{align*}
 &  \left|\iiint_{\R^d\times\R^d\times\R^d}(K_i(x-y)-K_i(t-y))\,d \gamma_{i,\tau}^n(x,t)\,d \rhojtn(y)\right| \\
 & \ \leq \mathrm{Lip}(K_i)\iiint_{\R^d\times\R^d\times\R^d} |x-t|\,d \gamma_{i,\tau}^n(x,t)\,d \rhojtn(y)\\
 & \ \leq  \mathrm{Lip}(K_i) \mW_2(\rhoitn,\rhoitnn) \leq \frac{1}{4\tau}\mW_2^2(\rhoitn,\rhoitnn) + C\tau,
\end{align*}
for some constant $C>0$ independent of $\tau$. Substituting the above estimate into \eqref{eq:scheme1}, we get
\begin{equation*}
\begin{split}
\frac{1}{4\tau}\mW_2^2(\rrhotn,\rrhotnn)&\le\int_{\Rd}A(\rhoutn,\rhodtn)\,dx-\int_{\Rd}A(\rhoutnn,\rhodtnn)\,dx\quad+\\&+\sum_{i=1}^{2}\frac{1}{2}\left(\int_{\Rd}H_i * \rhoitn\,d\rhoitn-\int_{\Rd}H_i * \rhoitnn\,d\rhoitnn\right) + C\tau.
\end{split}
\end{equation*}
With the notation in Remark \ref{splitfunc}, let us rewrite the previous estimate as
\begin{equation}
\frac{1}{4\tau}\mW_2^2(\rrhotn,\rrhotnn)\le\tilde{\mF}[\rrhotn]-\tilde{\mF}[\rrhotnn]+C\tau,
\end{equation}
which implies
\begin{equation}\label{ftilde}
\tilde{\mF}[\rrhotn]\le\tilde{\mF}[\rrho_0]+CT, \qquad \forall n\in\mathbb{N}.
\end{equation}
Summing over $k$ from $m$ to $n$ with $m<n$ we get

\begin{equation}\label{telescopic}
\begin{split}
\frac{1}{4\tau}\sum_{k=m}^{n}\mW_2^2(\rrho_{\tau}^{k},\rrho_{\tau}^{k+1})&\le\int_{\Rd}A(\rhout^m,\rhodt^m)\,dx-\int_{\Rd}A(\rhoutnn,\rhodtnn)\,dx\\&+\sum_{i=1}^{2}\frac{1}{2}\left(\int_{\Rd}H_i * \rhoit^m\,d\rhoit^m-\int_{\Rd}H_i * \rhoitnn\,d\rhoitnn\right)\\&+ C(n-m+1)\tau.
\end{split}
\end{equation}\noindent
Thanks to the estimates \eqref{ftilde} and \eqref{telescopic} we obtain
\begin{equation}\label{firstdistance}
\mW_2^2(\rrhot(0),\rrhot(t))\le C(\rrho_0)n\tau+Cn^2\tau^2+m_2(\rrho_0)+C(T,\alpha)
\end{equation}
by using H\"{o}lder inequality and (weighted) Young inequality, taking into account that $A\ge0$, the control from below on $H_i$ and Remark \ref{momp}.
The estimate \eqref{firstdistance} and Remark \ref{momp} allow us to conclude that the second moment of $\rrho_{\tau}(t)$ is uniformly bounded on compact time intervals.\\
Thanks to \eqref{ftilde} and \eqref{firstdistance}, we can improve the estimate \eqref{telescopic} in the following way: 
\begin{equation}\label{telescopic1}
\sum_{k=m}^{n-1}\mW_2^2(\rrho_{\tau}^{k},\rrho_{\tau}^{k+1})\le\overline{C}\tau+C(n-m)\tau^2,
\end{equation}
where $\overline{C}=\overline{C}(\rrho_0,T,\alpha)$.
Now, let us consider $0\le s<t$ such that $s\in((m-1)\tau,m\tau]$ and $t\in((n-1)\tau,n\tau]$ (which implies $|n-m|<\frac{|t-s|}{\tau}+1$); by Cauchy-Schwartz inequality, \eqref{firstdistance} and \eqref{telescopic1} we obtain

\begin{equation}\label{holdercont}
\begin{split}
\mW_2(\rrhot(s),\rrhot(t))&\le\sum_{k=m}^{n-1}\mW_2(\rrho_{\tau}^{k},\rrho_{\tau}^{k+1})\le\left(\sum_{k=m}^{n-1}\mW_2^2(\rrho_{\tau}^{k},\rrho_{\tau}^{k+1})\right)^{\frac{1}{2}}|n-m|^{\frac{1}{2}}\\&\le c\ \sqrt{1+T}\left(\sqrt{|t-s|}+\sqrt{\tau}\right),
\end{split}
\end{equation}\noindent
where $c$ is a positive constant.
This means that $\rrhot$ is $\frac{1}{2}$-H\"{o}lder equi-continuous (up to a negligible error of order $\sqrt{\tau}$) and then we obtain the uniform narrow compactness of $\rrhot$ on compact time intervals by using a refined version of Ascoli-Arzel\`{a}'s theorem (see \cite{AGS}, Section 3).
\end{proof}

\begin{rem}\label{lmbound}
\emph{An important consequence following from \eqref{exD2}, \eqref{ftilde}, and \eqref{firstdistance} is that
\begin{equation}
\sup_{t\in[0,T]}\left[||\rhoutn||_{L^{m_1}(\Rd)}^{m_1} + ||\rhodtn||_{L^{m_2}(\Rd)}^{m_2}\right]\le\tilde{\mF}[\rrho_0] + C(T,\rrho_0), \qquad \forall n\in\mathbb{N}.
\end{equation}}
\end{rem}\noindent
Please notice that all the measures involved in the above calculations are absolutely continuous with respect to Lebesgue measure, which comes as a trivial consequence of the fact that these measures minimize the JKO scheme and therefore the diffusive term of the functional $\mF$ must be finite.

\subsection{Flow interchange}
The uniform-in-time narrow convergence of our approximating sequence $\rrhot$ is not strong enough to pass to the limit and get the weak formulation of our system \eqref{sist}. Due to the nonlinearities involved, we need to prove strong compactness in suitable $L^p$ spaces for $\rhou$ and $\rhod$. This problem can be solved by means of the so-called ``flow interchange'' technique developed by Matthes, McCann and Savar\'{e} (\cite{MMCS}) and used later on by several authors. The seminal idea behind this technique is that \textit{the dissipation
of one functional along the gradient flow of another functional equals the dissipation of the second
functional along the gradient flow of the first}. By computing variations as solutions to a certain gradient flow, one can use the ``Evolution Variational Inequality'' (EVI) to obtain refined estimates and the desired compactness. It is well-known that the heat equation can be regarded as a steepest descent of the opposite of the Boltzmann entropy (\cite{JKO}), i.e. $\int_{\Rd}\rho(x)\log\rho(x)\,dx$. This can be easily extended to the multi-species case by reformulating the problem in the product space $(\mpd\times\mpd,\mW_2)$. Hence, the system
\begin{equation}\label{hs}
\begin{cases}
\partial_{t}\eta_1=\Delta\eta_1 \\
\partial_{t}\eta_2=\Delta\eta_2
\end{cases}
\end{equation}
can be seen as the gradient flow of the functional
\begin{equation*}
\mE[\eta_1,\eta_2]=\int_{\Rd}[\eta_1(x)\log\eta_1(x)+\eta_2(x)\log\eta_2(x)]\,dx,
\end{equation*}
with respect to the 2-Wasserstein distance $\mW_2$.
\begin{rem}\label{controlbelowentropy}
\emph{Let us call the entropy $\mE(\rho)=\int_{\Rd}\rho(x)\log\rho(x)\,dx$. From the seminal work of Jordan, Kinderlehrer and Otto (\cite[Proposition 4.1]{JKO}), it comes out that the entropy is controlled from below by the second momentum $m_2(\rho)$, i.e.
$$
\mE(\rho)\ge -C(m_2(\rho) + 1)^{\beta},
$$
for every $\rho\in\mpda$, $\beta\in(\frac{d}{d+2},1)$ and $C<+\infty$, depending only on the space dimension $d$. We will use this inequality in order to have a uniform bound from below for the entropy.}
\end{rem}\noindent
We want to develop the flow interchange strategy using \eqref{hs}, so the auxiliary functional we are going to use is
\begin{equation}\label{af}
\mE[\eta_1,\eta_2]=
\begin{cases}
\int_{\Rd}[\eta_1(x)\log\eta_1(x)+\eta_2(x)\log\eta_2(x)]\,dx, &\eta_1\log\eta_1,\eta_2\log\eta_2\in L^1(\Rd);\\
+\infty & \text{otherwise}.
\end{cases}
\end{equation}\normalsize
As we recalled in the section \ref{pre}, the entropy functional possesses a $0$-flow given by the heat semigroup; therefore, let us consider $\bm{\nu}=(\nu_1,\nu_2)\in\mpd^2$ such that $\mE(\nu_1,\nu_2)<+\infty$ and let us denote by $\bm{S}_{\mE}=(\su,\sd)$ the $0$-flow associated to $\mE$. In particular, denoting by $$
\su^t\nu_1:=\eta_1(t,\cdot)\quad \text{and}\quad \sd^t\nu_2:=\eta_2(t,\cdot),
$$
we have that $\bm{\eta}(t,\cdot)=(\eta_1(t,\cdot),\eta_2(t,\cdot))$ is the solution at time $t$ of the system \eqref{hs} coupled with an initial value $\bm{\nu}$ at $t=0$. For every $\rrho\in\mpd^2$ and every given $\mmu\in\mpd^2$, let us define the dissipation of $\mF$ along $\bm{S}_{\mE}$ by
$$
\bm{D}_{\mE}(\rrho|\mmu):=\limsup_{s\downarrow0}\left\{\frac{\mF[\rrho|\mmu]-\mF[\bm{S}_{\mE}^s\rrho|\mmu]}{s}\right\}.
$$
\begin{thm}\label{tfi}
There exists a constant C depending on T and $\rrho_0$ (not on $\tau$) such that the piecewise constant interpolation $\rrhot$ satisfies
\begin{equation}\label{bound}
||\rhout^{m_1/2}||_{L^2(0,T;H^1(\Rd))}+||\rhodt^{m_2/2}||_{L^2(0,T;H^1(\Rd))}\le C \qquad \text{for all}\ T>0.
\end{equation}
\end{thm}
\begin{proof}
From the definition of the sequence $\{\rrhotn\}_{n\in\mathbb{N}}$, for all $s>0$ we have that
$$
\frac{1}{2\tau}\mW_2^2(\rrhotnn,\rrhotn)+\mF[\rrhotnn|\rrhotn]\le\frac{1}{2\tau}\mW_2^2(\bm{S}_{\mE}^s\rrhotnn,\rrhotn)+\mF[\bm{S}_{\mE}^s\rrhotnn|\rrhotn],
$$
which gives, dividing by $s>0$ and passing to the $\limsup$ as $s\downarrow0$,
\begin{equation}\label{controlf}
\tau\bm{D}_{\mE}\mF(\rrhotnn|\rrhotn)\le\frac{1}{2}\frac{d^+}{dt}\bigg(\mW_2^2(\bm{S}_{\mE}^t\rrhotnn,\rrhotn)\bigg)\Big|_{t=0}\overset{\bm{(E.V.I.)}}{\le}\mE[\rrhotn]-\mE[\rrhotnn].
\end{equation}
In the last inequality we have used the well-known equivalence between displacement convexity and the existence of the E.V.I., see e.g. \cite{DS}. Now, let us focus on the left hand side of \eqref{controlf}. First of all, note that:
\begin{equation}\label{integralformofdis}
\begin{split}
\bm{D}_{\mE}\mF(\rrhotnn|\rrhotn)&=\limsup_{s\downarrow0}\left\{\frac{\mF[\rrhotnn|\rrhotn]-\mF[\bm{S}_{\mE}^s\rrhotnn|\rrhotn]}{s}\right\}\\&=\limsup_{s\downarrow0}\int_0^1\left(-\frac{d}{dz}\Big|_{z=st}\mF[\bm{S}_{\mE}^{z}\rrhotnn|\rrhotn]\right)\,dt.
\end{split}
\end{equation}
So, let us compute the time derivative inside the above integral, using integration by parts and keeping in mind the $C^\infty$ regularity of the solution to the heat equation:
\begin{equation}\label{deriv}
\begin{split}
\frac{d}{dt}\mF[\bm{S}_{\mE}^t\rrhotnn|\rrhotn]=&-\int_{\R^d}A_{\rhou\rhou}(\su^t\rhoutnn,\sd^t\rhodtnn)|\nabla(\su^t\rhoutnn)|^2\,dx\\&-\int_{\R^d}A_{\rhou\rhod}(\su^t\rhoutnn,\sd^t\rhodtnn)\nabla(\su^t\rhoutnn)\nabla(\sd^t\rhodtnn)\,dx\\&-\int_{\R^d}A_{\rhod\rhou}(\su^t\rhoutnn,\sd^t\rhodtnn)\nabla(\sd^t\rhodtnn)\nabla(\su^t\rhoutnn)\,dx\\&-\int_{\R^d}A_{\rhod\rhod}(\su^t\rhoutnn,\sd^t\rhodtnn)|\nabla(\sd^t\rhodtnn)|^2\,dx\\&-\sum_{i=1}^2\int_{\Rd}\int_{\Rd}\nabla H_i(x-y)\si^t\rhoitnn(y)\nabla\si^t\rhoitnn(x)\,dy\,dx\\&-\sum_{i\neq j}\int_{\Rd}\int_{\Rd}\nabla K_i(x-y)\rhojtn(y)\nabla\si^t\rhoitnn(x)\,dy\,dx.
\end{split}
\end{equation}
By using the crucial coerciveness assumption (D3) for the terms involving the diffusion function $A$ and the distributional control of $\Delta H_i$ and $\Delta K_i$ in assumption (HK2), we obtain
\begin{equation}\label{controld}
\frac{d}{dt}\mF[\bm{S}_{\mE}^t\rrhotnn]\le-C_1\int_{\R^d}|\nabla(\su^t\rhoutnn)^{\frac{m_1}{2}}|^2+|\nabla(\sd^t\rhodtnn)^{\frac{m_2}{2}}|^2\,dx \ + \ \bar{C}.
\end{equation}
Note that the reconstruction of the gradient terms in \eqref{controld} involves powers of order $m_1-2$ and $m_2-2$, and these two exponents may be negative. This is not a problem because $\bm{S}_{\mE}^t\rrhotnn$ solves a decoupled system of heat equations, therefore both components of $\bm{S}_{\mE}^t\rrhotnn$ are strictly positive everywhere. The above inequality, together with \eqref{integralformofdis}, implies:
\begin{equation}\label{controld+}
\begin{split}
\bm{D}_{\mE}\mF(\rrhotnn|\rrhotn)\ge C_1\liminf_{s\downarrow0}\int_0^1\int_{\R^d}|\nabla(\su^{st}\rhoutnn)^{\frac{m_1}{2}}|^2+|\nabla(\sd^{st}\rhodtnn)^{\frac{m_2}{2}}|^2\,dx\,dt\ - \bar{C}.
\end{split}
\end{equation}
\normalsize As a consequence of \eqref{controlf} and \eqref{controld+} we obtain that
\begin{equation}
\tau\,C_1\,\liminf_{s\downarrow0}\int_0^1\int_{\R^d}|\nabla(\su^{st}\rhoutnn)^{\frac{m_1}{2}}|^2+|\nabla(\sd^{st}\rhodtnn)^{\frac{m_2}{2}}|^2\,dx\,dt\le\mE[\rrhotn]-\mE[\rrhotnn] +\bar{C}\tau.
\end{equation}\noindent
Let us recall that $\rrhotn\in L^{m_1}(\Rd)\times L^{m_2}(\Rd)$ for every $n\in\mathbb{N}$ (Remark \ref{lmbound}), so $(\rhoitn)^{\frac{m_i}{2}}\in L^2(\Rd)$ for all $n\in\mathbb{N}$ and for $i=1,2$. Then it is well-known from the heat equation's theory that $(\bm{S}_{\mE_i}^t\rhoitnn)^{\frac{m_i}{2}}$ converges to $(\rhoitnn)^{\frac{m_i}{2}}$ in $L^2(\R^d)$ as $t\downarrow0$ for $i=1,2$; thus, by weak lower semicontinuity we have:
\begin{equation}
\tau\,C_1\,\int_{\R^d}\left[|\nabla(\rhoutnn)^{\frac{m_1}{2}}|^2+|\nabla(\rhodtnn)^{\frac{m_2}{2}}|^2\right]\,dx\le\mE[\rrhotn]-\mE[\rrhotnn]\ + \bar{C}\tau.
\end{equation}
Since $x\log x<x^m$ for $m>1$ and for every $x\in\R_+$, using Remark \ref{lmbound} we immediately get 
$$
\mE[\rrhotn]<||\rhoutn||_{L^{m_1}(\Rd)}^{m_1}+||\rhodtn||_{L^{m_2}(\Rd)}^{m_2}\le\tilde{\mF}[\rrho_0] + C(T).
$$
Moreover, a combination of Remarks \ref{momp} and \ref{controlbelowentropy} and the estimate \eqref{firstdistance} immediately gives a uniform boundedness from below for $\mE[\rrhotn]$ (taking into account that the initial condition has finite second moment). If we sum over $n$ from 0 to $N-1$, we obtain
$$
\int_0^T\int_{\R^d}\left[|\nabla(\rhout)^{\frac{m_1}{2}}|^2+|\nabla(\rhodt)^{\frac{m_2}{2}}|^2\right]\,dx\,dt\le\mE[\bm{\rrho}_0]-\mE[\bm{\rrho}_{\tau}^N]+\bar{C}T\le C(T,\rrho_0),
$$
which allows us to conclude the proof, since we can use the estimate in Remark \ref{lmbound} to get
\begin{equation}\label{h1bound}
\int_0^T\left[||\rhout(t,\cdot)^{\frac{m_1}{2}}||_{H^1(\Rd)}^2+||\rhodt(t,\cdot)^{\frac{m_2}{2}}||_{H^1(\Rd)}^2 \right]\,dt\le C(T,\rrho_0).
\end{equation}
\end{proof}
\noindent
We now collect the results in Proposition \ref{propaa} and Theorem \ref{tfi} to prove strong $L^p$ compactness of $\rrhot$.
\begin{cor}\label{cfi}
The sequence $\rrho_{\tau_{k}}:[0,+\infty[\longrightarrow\mpd^2$ obtained in Proposition \ref{propaa} converges to $\tilde{\rrho}$ strongly in 
$$L^{m_1}(]0,T[\times\R^d)\times L^{m_2}(]0,T[\times\R^d),$$ 
for every $T>0$.
\end{cor}
\begin{proof}
We exploit Theorem \ref{aulirs} with $X=L^{m_1}(\Rd)\times L^{m_2}(\R^d)$, $g=\mW_2$ and the functional $\mathcal{N}$ defined by
\begin{equation*}
\mathcal{N}(\rrho)=
\begin{cases}
||\rhou^{\frac{m_1}{2}}||_{H^1(\Rd)}+||\rhod^{\frac{m_2}{2}}||_{H^1(\Rd)} + \int_{\R^d}|x|^2 (\rhou(x)+\rhod(x))\, dx, & \text{if both $\nabla\rhou^{\frac{m_1}{2}}\in L^2(\Rd)$}\\ 
 & \text{and $\nabla\rhod^{\frac{m_2}{2}}\in L^2(\Rd)$} ;\\
+\infty & \text{otherwise}.
\end{cases}
\end{equation*}
The lower semi-continuity of $\mathcal{N}$ on $X$ can be proven by adapting the proof of \cite[Lemma A.1]{DFM}. For any fixed $c>0$, let us set $A_c:=\{\rrho\in L^{m_1}(\Rd)\times L^{m_2}(\Rd) \,:\,\,\mathcal{N}(\rrho)\le c\}$ a sub-level of $\mathcal{N}$ and prove that it is relatively compact in $L^{m_1}(\Rd)\times L^{m_2}(\Rd)$. Setting $B_c:=\{\bm{\eta}=(\rhou^{\frac{m_1}{2}},\rhod^{\frac{m_2}{2}})\,:\,\,\rrho\in A_c\}$, we prove that $B_c$ is relatively compact in $L^2(\Rd)^2$, since the map $\iota: L^2(\Rd)^2\to L^{m_1}(\Rd)\times L^{m_2}(\Rd)$ with $\iota(\eta)=(\eta_1^{\frac{2}{m_1}},\eta_2^{\frac{2}{m_2}})$ is continuous and $A_c=\iota(B_c)$.
Every sub-level $B_c$ can be easily proven to be strongly relatively compact in $L^2(\Rd)$ in view of Riesz-Frechet-Kolmogorov Theorem, thanks to the uniform continuity estimate
\begin{align*}
& \int_{\R^d}|\eta_i(x+h)-\eta_i(x)|^2\,dx = \int_{\R^d}\left|\int_0^1 \frac{d}{d\tau}\eta_i(x+\tau h)\,d\tau\right|^2 \,dx = \int_{\R^d}\left|\int_0^1 h \cdot \nabla \eta_i(x+\tau h)\,d\tau\right|^2 \,dx\\
& \ \leq |h|^2\int_{\R^d}\int_0^1 |\nabla\eta_i(x+\tau h)|^2\,d\tau \,dx = |h|^2\int_{\R^d}\|\nabla \eta_i\|_{L^2(\R^d)}^2
\end{align*}
and the uniform integrability at infinity (using H\"older inequality)
\begin{align*}
& \int_{|x|\geq R}\eta_i(x)^2 \,dx\leq \frac{1}{R^{2\delta}}\int_{\R^d}|x|^{2\delta}\rho_i(x)^{m_i}\,dx\leq \frac{1}{R^{2\delta}}\left(\int_{\R^d}|x|^2\rho_i(x)\,dx\right)^\delta \left(\int_{\R^d}\rho_i(x)^{\frac{m_i-\delta}{1-\delta}}\,dx\right)^{1-\delta}
\end{align*}
in which $\delta$ can be chosen in $(0,1)$ such that $(m_i-\delta)/(1-\delta)= p_i$, with
$p_i\in\left(\max\left\{2,m_i\right\},+\infty\right)$ for $d=1,2$, such that  $\max\left\{2,m_i\right\}<p_i<\frac{2d}{d-2}$ for $d>2$, with those requirements implied by the Gagliardo-Nirenberg inequality
$$
\|\eta_i\|_{L^{p_i}}\le C\|\nabla \eta_i\|_{L^2}^{\theta_i}\|\eta_i\|_{L^2}^{1-\theta_i},\qquad \theta_i=\frac{(p_i-2)d}{2 p_i},
$$
which guarantees that $\|\eta_i\|_{L^{p_i}}$ is finite.

\noindent
Moreover, denoting $U:=\{\rrho_{\tau_{k}}\,:\,\,k\in\mathbb{N}\}$, the hypothesis \eqref{hprossav} of Theorem \ref{aulirs} is satisfied thanks to the previous Theorem \ref{tfi} and to the $\tau$-uniform approximate H\"{o}lder continuity \eqref{holdercont} (see \cite{DFM}). Hence, we have a subsequence $(\tau_{k'})\subseteq(\tau_{k})$ such that $\rrho_{\tau_{k'}}$ converges in measure (as a function of time $t\in[0,T]$ with values in $X$) to some limit $\rrho'$, which has to be $\tilde{\rrho}$ due to the narrow convergence of $\rrho_{\tau_{k}}$ to $\tilde{\rrho}$ uniformly in time. Actually we can state that the whole sequence $\{\rrho_{\tau_{k}}\}$ converges in measure to $\tilde{\rrho}$ and, as a consequence, we have almost everywhere convergence (up to a subsequence). Using Remark \ref{lmbound} and a suitable $L^p$ interpolation with respect to $t$, we obtain the strong convergence of $\rrho_{\tau_{k}}$ to $\tilde{\rrho}$ in $L^{m_1}(]0,T[\times\R^d)\times L^{m_2}(]0,T[\times\R^d)$.
\end{proof}

With the help of some standard interpolation inequalities we can actually improve the integrability exponent for the strong convergence of $\rrho_{\tau_{k}}$.

\begin{cor}\label{cfi1}
The sequence $\rrho_{\tau_{k}}:[0,+\infty[\longrightarrow\mpd^2$ in Corollary \ref{cfi} converges to $\tilde{\rrho}$ in 
$$L^{\alpha_1}(]0,T[\times\R^d)\times L^{\alpha_2}(]0,T[\times\R^d),$$
for every $T>0$, provided $\alpha_1$ and $\alpha_2$ satisfy
\begin{equation}\label{improved_Lp}
\alpha_i<M_i:=\frac{m_i\, (d+2)}{d} 
\end{equation}
with $i=1,2$. 
\end{cor}
\begin{proof}
For simplicity in the notation, we shall denote the subsequence $\rrho_{\tau_k}$ by $\rrho_\tau$. As a consequence of the previous Corollary and Remark \ref{lmbound} we get
\begin{equation*}
\sup_{t\in[0,T]}\left[||\rhout(t,\cdot)-\tilde{\rho_1}(t,\cdot)||_{L^{m_1}(\Rd)}+||\rhodt(t,\cdot)-\tilde{\rho_2}(t,\cdot)||_{L^{m_2}(\Rd)}\right]\le C(T,\rrho_0),
\end{equation*}
which implies
\begin{equation}\label{convsigma}
\int_0^T\left[||\rhout(t,\cdot)-\tilde{\rho_1}(t,\cdot)||_{L^{m_1}(\Rd)}^{\sigma_1}+||\rhodt(t,\cdot)-\tilde{\rho_2}(t,\cdot)||_{L^{m_2}(\Rd)}^{\sigma_2}\right]\,dt\underset{\tau\to0}{\longrightarrow} 0,
\end{equation}
for $\sigma_1\ge m_1$ and $\sigma_2\ge m_2$. In case one of the two exponents $\sigma_i$ is smaller than $m_i$ then we can proceed as follows for an arbitrary $\epsilon>0$:
\begin{equation}\label{32epsilon}
\int_0^T ||\rhoit(t,\cdot)-\tilde{\rho_i}(t,\cdot)||_{L^{m_i}(\Rd)}^{\sigma_i}\, dt \le \left(\int_0^T ||\rhoit(t,\cdot)-\tilde{\rho_i}(t,\cdot)||_{L^{m_i}(\Rd)}^{m_i+\epsilon}\, dt\right)^{\frac{\sigma_i}{m_i+\epsilon}}T^{1-\frac{\sigma_i}{m_i+\epsilon}}
\end{equation}
and we get the same conclusion as in \eqref{convsigma} for all $\sigma_1,\sigma_2>0$. In order to obtain a refined convergence, we use the Gagliardo-Nirenberg inequality
$$
||f||_{L^p}\le C||\nabla f||_{L^r}^{\theta}||f||_{L^q}^{1-\theta},
$$ 
\noindent
where $1\le q,r\le +\infty$, $0<\theta<1$ and $p$ is such that $\frac{1}{p}=\theta(\frac{1}{r}-\frac{1}{d})+(1-\theta)\frac{1}{q}$. For $i=1,2$, we set $p_i=\frac{2\alpha_i}{m_i}$, $q_i=r_i=2$, and we get
$$
\theta_i=\frac{(\alpha_i-m_i)d}{2\alpha_i}.
$$
We observe that requiring $\theta_i<1$ yields no restrictions in case $d=1,2$ and it requires $\alpha_i<\frac{m_i d}{d-2}$ in case $d>2$. Clearly, here we are implicitly assuming $\alpha_i\ge m_i$. The strong convergence of $\rho_{i,\tau_{k}}$ in $L^{\alpha_i}$ for $\alpha_i<m_i$ is a straightforward consequence of corollary \ref{cfi} and $L^p$ interpolation.
Using even the H\"{o}lder inequality we obtain
\begin{equation*}
\begin{split}
&||\rho_{i,\tau}^{\frac{m_i}{2}}-\tilde{\rho}_i^\frac{m_i}{2}||_{L_{t,x}^{p_i}}^{p_i}=\int_0^T||\rho_{i,\tau}^{\frac{m_i}{2}}(t,\cdot)-\tilde{\rho}_i^\frac{m_i}{2}(t,\cdot)||_{L^{p_i}}^{p_i}\,dt\\&\le C\int_0^T||\nabla\rho_{i,\tau}^{\frac{m_i}{2}}(t,\cdot)-\nabla\tilde{\rho}_i^\frac{m_i}{2}(t,\cdot)||_{L^2}^{p_i\theta_i}||\rho_{i,\tau}^{\frac{m_i}{2}}(t,\cdot)-\tilde{\rho}_i^\frac{m_i}{2}(t,\cdot)||_{L^2}^{p_i(1-\theta_i)}\,dt\\&\le C\left(\int_0^T||\nabla\rho_{i,\tau}^{\frac{m_i}{2}}(t,\cdot)-\nabla\tilde{\rho}_i^\frac{m_i}{2}(t,\cdot)||_{L^2}^{2}\,dt\right)^{\frac{p_i\theta_i}{2}}\left(\int_0^T||\rho_{i,\tau}^{\frac{m_i}{2}}(t,\cdot)-\tilde{\rho}_i^\frac{m_i}{2}(t,\cdot)||_{L^2}^{\gamma_i}\,dt\right)^{\frac{m_i-\alpha_i\theta_i}{m_i}}.
\end{split}
\end{equation*}

where $$\gamma_i=\frac{(1-\theta_i)2\alpha_i}{m_i-\alpha_i\theta_i}.$$ Thanks to the result in Theorem \ref{tfi}, the first term at the right-hand side above is uniformly bounded. Motivated by $\rho_{i,\tau_k}$ converging strongly in $L^{m_i}_{x,t}$ to $\tilde{\rho}_{i}$, \eqref{convsigma} and \eqref{32epsilon} imply the assertion provided $\gamma_i>0$, which yields the condition \eqref{improved_Lp} above.
\end{proof}

\subsection{Consistency of the scheme: convergence to weak solutions}
We are now ready to deal with the consistency of the scheme, i.e. with proving that the strong limit $\tilde{\rrho}$ is a weak solution to \eqref{sist} in the sense of definition \ref{defweaksolution}. As usual, see \cite{JKO}, this task is performed by writing down the Euler-Lagrange equations related to the scheme \eqref{jko}.

\begin{thm}
The approximating sequence $\rrho_{\tau_k}$ converges to a weak solution $\tilde{\rrho}$ to \eqref{sist}.
\end{thm}
\begin{proof}
We split the proof into several steps to improve its readability.

\smallskip
\noindent
\textbf{Step 0: perturbation of the the JKO optimiser.} Consider two consecutive steps in the semi-implicit JKO scheme \eqref{jko}, i.e. $\rrhotn$, $\rrhotnn$, and let us proceed by perturbing the first component of $\rrhotnn$ in the following way  
\begin{equation}\label{eq:perturbation}
 \rrho^\ee = (\rho_1^\ee,\rho_2^\ee)=(P_\#^\ee\rhoutnn,\rhodtnn),
\end{equation}
where $P^\ee=\id+\ee\zeta$, for some $\zeta\in C_c^\infty(\Rd;\Rd)$ and $\ee\ge0$. From the minimizing property of $\rrhotnn$ we have 
\begin{equation}\label{optimality}
0\leq\frac{1}{2\tau}\left[\mW_2^2(\rrhotn,\rrho^\ee)- \mW_2^2(\rrhotn, \rrhotnn)\right]+\mF[\rrho^\ee|\rrhotn]-\mF[\rrhotnn|\rrhotn].
\end{equation}
We now analyse the several terms contained in \eqref{optimality}.

\smallskip
\noindent
\textbf{Step 1: the nonlocal interaction terms.}
The self interaction term involving $H_2$ gives a null contribution in the difference $\mF[\rrho^\ee|\rrhotn]-\mF[\rrhotnn|\rrhotn]$, whereas the $H_1$-self-interaction terms give 

\begin{equation}\label{integral1}
\begin{split}
&\frac{1}{2}\int_{\R^d}H_{1}\ast\rho_{1}^\ee d\rho_{1}^\ee-\frac{1}{2}\int_{\R^d}H_{1}\ast\rhoutnn  d\rhoutnn\\&=\frac{1}{2}\int_{\R^{2d}}\left[H_{1}(P^\ee(x)-P^\ee(y))-H_{1}(x-y)\right]\rhoutnn(y) \rhoutnn(x)\,dy\,dx\\&=\frac{1}{2}\int_{\R^{2d}}\left[H_{1}(x-y +\ee(\zeta(x)-\zeta(y)))-H_{1}(x-y)\right]\rhoutnn(y)\rhoutnn(x)\,dy\,dx.
\end{split}
\end{equation}\noindent
Now, from the assumptions on $H_1$ we get
\begin{equation}\label{convergH}
\frac{H_{1}(x-y +\ee(\zeta(x)-\zeta(y)))-H_{1}(x-y)}{\ee} \rightarrow \nabla H_1(x-y)\cdot(\zeta(x)-\zeta(y))
\end{equation}
as $\ee\to 0$ for all $(x,y)\in \R^{2d}$. By means of Egorov's theorem, for every $\sigma>0$ there exists $B_\sigma\subset\Rdd$ measurable such that $$\int\int_{B_\sigma}\rhoutnn(y)\rhoutnn(x)\,dx\,dy<\sigma$$ and the convergence \eqref{convergH} is uniform on $\Rdd\setminus B_{\sigma}$, while in $B_\sigma$ the control on the gradient $|\nabla H_i(x)|\le C(1+|x|)$ in assumption (H1) allows us to neglect the integral on $B_\sigma$ in the limit-integral interchange, so we get
\begin{equation*}
\begin{split}
&\int_{\R^{2d}}\left(\frac{H_{1}(x-y +\ee(\zeta(x)-\zeta(y)))-H_{1}(x-y)}{\ee}\right)\rhoutnn(y) \rhoutnn(x)\,dy\,dx\\&\quad\rightarrow \int_{\R^{2d}}\nabla H_1(x-y)\cdot(\zeta(x)-\zeta(y))\rhoutnn(y) \rhoutnn(x)\,dy\,dx.
\end{split}
\end{equation*}
Therefore, by Taylor expansion the last term in \eqref{integral1} can be written as
\begin{equation*}
\frac{\ee}{2}\int_{\R^{2d}}\nabla H_{1}(x-y)\cdot\left(\zeta(x)-\zeta(y)\right)\rhoutnn(y) \rhoutnn(x)\,dy\,dx+o(\ee).
\end{equation*}
Let us now compute the terms in $\mF[\rrho^\ee|\rrhotn]-\mF[\rrhotnn|\rrhotn]$ involving the cross-interaction potentials. Once again, as the perturbation of the identity is directed only in the first component, the term involving $K_2$ cancels out, and we are left with the contribution
\begin{equation*}
\begin{split}
&\int_{\Rd}K_{1}\ast\rhodtn d\rho_{1}^\ee-\int_{\Rd}K_{1}\ast\rhodtn d\rhoutnn\\&=\int_{\R^{2d}}\left(K_{1}(x+\ee\zeta(x)-y)-K_{1}(x-y)\right)\rhodtn(y) \rhoutnn(x)\,dy\,dx\\&=\ee\int_{\R^{2d}}\nabla K_{1}(x-y)\cdot\zeta(x)\rhodtn(y)\rhoutnn(x)\,dy\,dx+o(\ee),
\end{split}
\end{equation*}
where the last step can be justified as before. Notice that no symmetrization can be performed here to compensate a possible discontinuity of $\nabla K_1$ at zero, that is why we need $\nabla K$ to be continuous everywhere, which is guaranteed by assumption (K1).

\smallskip
\noindent
\textbf{Step 2: the diffusion term.} 
We define $B:[0,+\infty)^2\rightarrow \R$ as 
\[B(\eta_1,\eta_2):= A\left(\eta_1^{\frac{2}{m_1}},\eta_2^{\frac{2}{m_2}}\right).\]
With this notation, the difference between the diffusion terms in \eqref{optimality} can be rewritten as
\begin{equation}\label{diffusiondiff}
\begin{split}
& \int_{\R^d}A\left(\rho_1^\ee(x),\rho_2^\ee(x)\right)\,dx-\int_{\Rd}A(\rhoutnn(x),\rhodtnn(x))\,dx\\
& \quad = \int_{\R^d} B\left((\rho_1^\ee(x))^{m_1/2},(\rho_2^\ee(x))^{m_2/2}\right)\, dx - \int_{\R^d} B\left(\rhoutnn(x)^{m_1/2}, \rhodtnn(x)^{m_2/2}\right)\, dx.
\end{split}
\end{equation}
For simplicity, we shall denote for $i=1,2$,
\begin{align*}
& \etaitnn(x):= (\rhoitnn(x))^{m_i/2}\,\qquad \eta_i^\ee(x):= (\rho_i^\ee(x))^{m_i/2}.
\end{align*}
Hence, the first term of the above difference can be written as follows by using the definition of push-forward, the change-of-variables formula and the Taylor-Lagrange expansion of $B$ of order $2$:
\begin{equation}\label{diffusiontaylor}
\begin{split}
&\int_{\R^d}B\left(\eta_1^\ee(x),\eta_2^\ee(x)\right)\,dx=\int_{\R^d}B(\eta_1^\ee(x),\etadtnn(x))\,dx\\&=\int_{\R^d}B\left(\frac{\etautnn(x)}{(\det(\nabla P^\ee(x)))^{m_1/2}},\etadtnn(P^\ee(x))\right)\det(\nabla P^\ee(x))\,dx\\&=\int_{\Rd}B(\etautnn(x),\etadtnn(x))\det(\nabla P^\ee(x))\,dx\\&+\int_{\R^d}B_{\eta_1}(\etautnn (x),\etadtnn(x))\etautnn(x)[1-(\det(\nabla P^\ee(x)))^{\frac{m_1}{2}}]\det(\nabla P^\ee(x))^{1-\frac{m_1}{2}}\,dx\\&+\int_{\Rd}B_{\eta_2}(\etautnn(x),\etadtnn(x))[\etadtnn(P^\ee(x))-\etadtnn(x)]\det(\nabla P^\ee(x))\,dx\ + R.
\end{split}
\end{equation}
Here the remainder term $R$ is defined as follows for intermediate points $\bar{\etau}^\ee(x)$ (resp. $\bar{\etad}^\ee(x)$) between $\etautnn(x)$ and $\etautnn(x)/(\det(\nabla P^\ee(x)))^{m_1/2}$ ($\etadtnn\circ P^\ee(x)$ and $\etadtnn(x)$ resp.):
\begin{equation*}
\begin{split}
R&=\frac{1}{2}\int_{\Rd}B_{\etau\etau}(\bar{\etau}^\ee,\bar{\etad}^\ee)[\etautnn(x)]^2 [1-(\det(\nabla P^\ee(x)))^{\frac{m_1}{2}}]^2(\det(\nabla P^\ee(x)))^{1-m_1}\,dx\\&+\int_{\Rd}B_{\etau\etad}(\bar{\etau}^\ee,\bar{\etad}^\ee)\etautnn(x)[\etadtnn(P^\ee(x))-\etadtnn(x)][1-(\det(\nabla P^\ee(x)))^{\frac{m_1}{2}}]\det(\nabla P^\ee(x))^{1-\frac{m_1}{2}}\,dx\\&+\frac{1}{2}\int_{\Rd}B_{\etad\etad}(\bar{\etau}^\ee,\bar{\etad}^\ee)[\etadtnn(P^\ee(x))-\etadtnn(x)]^2\det(\nabla P^\ee(x))\,dx.
\end{split}
\end{equation*}
Recalling the formula $\det(\nabla P^\ee(x))=1+\ee\text{div}\zeta(x)+o(\ee)$ and by Taylor expanding $(1-z^{m_1/2})$ around $z=1$, we easily get
\[1-(\det(\nabla P^\ee(x)))^{\frac{m_1}{2}}=-\frac{m_1}{2}(\det(\nabla P^\ee(x))-1) + o(\ee)=-\frac{m_1\, \ee}{2}\text{div}\zeta(x)+o(\ee).\]
Moreover, we have
\[(\det(\nabla P^\ee(x)))^{1-m_1} = 1+o(\ee)=(\det(\nabla P^\ee(x)))^{1-\frac{m_1}{2}}.\]
Now, we claim that the remainder $R$ goes to zero faster than $\ee$ as $\ee\searrow 0$. This fact can be easily checked in view of the growth conditions (D2), of Lemma \ref{convlemma} (recall $\nabla \eta_{i,\tau}^{n+1} \in L^2$ in Theorem \ref{tfi}) which guarantees that the differences $\eta_{2,\tau}^{n+1}\circ P^\ee - \eta_{2,\tau}^{n+1}$ are $O(\ee)$ in $L^2$ as $\ee\searrow 0$, of the uniform $L^{\alpha_i}$ control of $\rho_{i,\tau}^{n+1}$ consequence of the Gagliardo-Nirenberg inequality in corollary \ref{cfi}, of the Lebesgue dominated convergence Theorem (recalling that $\zeta\in C^\infty_c(\R^d)$), and by using  Cauchy-Schwarz inequality several times. For the same reason, the term
\[\int_{\R^d} B_{\eta_2}(\etautnn(x),\etadtnn(x))[\etadtnn(P^\ee(x))-\etadtnn(x)](\ee\text{div}\zeta(x) +o(\ee))\, dx\]
in \eqref{diffusiontaylor} is $o(\ee)$ as $\ee\searrow 0$. This computation is quite standard and the details are left to the reader. Therefore, \eqref{diffusiondiff} becomes
\begin{equation*}
\begin{split}
&\int_{\R^d}B\left(\eta_1^\ee(x),\eta_2^\ee(x)\right)\,dx-\int_{\R^d}B(\etautnn(x),\etadtnn(x))\,dx\\&=\ee\int_{\R^d}B(\etautnn(x),\etadtnn(x))\text{div}\zeta(x)\,dx\\&-\frac{\ee m_1}{2}\int_{\R^d}B_{\eta_1}\left(\etautnn (x),\etadtnn(x)\right)\etautnn(x)\text{div}\zeta(x)\,dx\\&+\int_{\Rd}B_{\eta_2}(\etautnn(x),\etadtnn(x))[\etadtnn(P^\ee(x))-\etadtnn(x)]\,dx +o(\ee).
\end{split}
\end{equation*}
\smallskip

\noindent
\textbf{Step 3: the Wasserstein distance terms.} For the sake of completeness, we recall the standard computations to deal with the terms in \eqref{optimality} involving the Wasserstein distance. Brenier's Theorem allows to take $T:=T_{n}^{n+1}$ the optimal map between $\rhoutn$ and $\rhoutnn$ (see \cite{S,V1,V2}) and then
$$
W_2^2(\rhoutn,\rhoutnn)=\int_{\Rd}|x-T(x)|^2\rhoutn(x)\,dx,
$$
while
$$
W_2^2(\rhoutn,\rho_1^\ee)\le\int_{\R^d}|x-P^\ee(T(x))|^2\rhoutn(x)\,dx,
$$
since the map $P^\ee\circ T$ transports $\rhoutn$ into $\rhou^\ee$, but we do not know if it is optimal or not. Hence,
\begin{equation*}
\begin{split}
&\frac{1}{2\tau}\left(\mW_2^2(\rrhotn, \rrho^\ee)-\mW_2^2(\rrhotn, \rrhotnn)\right)=\frac{1}{2\tau}\left(W_2^2(\rhoutn,\rho_1^\ee)- W_2^2(\rhoutn,\rhoutnn)\right)\\
 & \leq\frac{1}{2\tau}\int_{\R^d}\left(|x-P^\ee(T(x))|^2 -|x-T(x)|^2\right)\rhoutn(x)\,dx\\
 & =\frac{1}{2\tau}\int_{\R^d}\left(|x-T(x)-\ee\zeta(T(x))|^2 -|x-T(x)|^2\right) \rhoutn(x)\,dx\\
 &=-\frac{\ee}{\tau}\int_{\R^d}(x-T(x))\cdot \zeta(T(x))\rhoutn(x)dx+o(\ee).
\end{split}
\end{equation*}
\smallskip

\noindent
\textbf{Step 4: Sending $\ee$ to zero.} Summing up all the contributions, dividing by $\ee$ and performing again the same computation with $\ee\le0$, we obtain for $\zeta=\nabla\varphi$
\begin{equation}\label{soldiscretaconeps}
\begin{split}
\frac{1}{\tau}&\int_{\R^d}(x-T(x))\cdot\nabla\varphi (T(x))\rhoutn(x)\,dx\\&=\int_{\R^d}B(\etautnn(x),\etadtnn(x))\Delta\varphi(x)\,dx\\&-\frac{m_1}{2}\int_{\R^d}B_{\eta_1}\left(\etautnn (x),\etadtnn(x)\right)\etautnn(x)\Delta\varphi(x)\,dx\\&+\int_{\Rd}B_{\eta_2}(\etautnn(x),\etadtnn(x))\frac{[\etadtnn(P^\ee(x))-\etadtnn(x)]}{\ee}\,dx\\& + \frac{1}{2}\int_{\R^{2d}}\nabla H_{1}(x-y)\cdot\left(\nabla\varphi(x)-\nabla\varphi(y)\right)\rhoutnn(y)\rhoutnn(x)\,dy\,dx\\
 & +\int_{\R^{2d}}\nabla K_{1}(x-y)\cdot\nabla\varphi(x)\rhodtn(y)\rhoutnn(x)\,dy\,dx+ o(\ee).
\end{split}
\end{equation}\noindent
Notice that when $\ee$ goes to $0$ we have that
\begin{equation*}
\begin{split}
&\int_{\Rd}B_{\eta_2}(\etautnn(x),\etadtnn(x))\frac{[\etadtnn(P^\ee(x))-\etadtnn(x)]}{\ee}\,dx\\&-\int_{\Rd}B_{\eta_2}(\etautnn(x),\etadtnn(x))\nabla\eta_{2,\tau}^{n+1}(x)\nabla\varphi(x)\,dx\,\to\,0,
\end{split}
\end{equation*}
thanks to assumption (D2) and Lemma \ref{convlemma} in the appendix, since $\eta_{2,\tau}^{n+1}\in H^1$. 
Hence, by sending $\ee$ to $0$ in \eqref{soldiscretaconeps} we obtain
\begin{equation}\label{finalcomp}
\begin{split}
\frac{1}{\tau}&\int_{\R^d}(x-T(x))\cdot\nabla\varphi (T(x))\rhoutn(x)\,dx\\&= \int_{\R^d}B(\etautnn(x),\etadtnn(x))\Delta\varphi(x)\,dx\\&-\frac{m_1}{2}\int_{\R^d}B_{\eta_1}\left(\etautnn (x),\etadtnn(x)\right)\etautnn(x)\Delta\varphi(x)\,dx\\&+\int_{\Rd}B_{\eta_2}(\etautnn(x),\etadtnn(x))\nabla\eta_{2,\tau}^{n+1}(x)\nabla\varphi(x)\,dx\\& + \frac{1}{2}\int_{\R^{2d}}\nabla H_{1}(x-y)\cdot\left(\nabla\varphi(x)-\nabla\varphi(y)\right)\rhoutnn(y)\rhoutnn(x)\,dy\,dx\\
 & +\int_{\R^{2d}}\nabla K_{1}(x-y)\cdot\nabla\varphi(x)\rhodtn(y)\rhoutnn(x)\,dy\,dx.
\end{split}
\end{equation}
By Taylor expanding $\varphi$ on $T(x)$ and using the definition of push-forward and the H\"{o}lder continuity estimate \eqref{holdercont}, we can rewrite the left-hand side of \eqref{finalcomp} as
$$
\int_{\R^d}(x-T(x))\cdot\nabla\varphi (T(x))\rhoutn(x)\,dx=\int_{R^d}\varphi(x)\left[\rhoutn(x)-\rhoutnn(x)\right]\,dx+O(\tau).
$$
Let $0\le t<s$ be fixed, with
$$
h=\left[\frac{t}{\tau}\right]\quad \text{and}\quad k=\left[\frac{s}{\tau}\right].
$$
By summing the equation \eqref{finalcomp}, using the chain rule on $B$, we have
\begin{equation*}
\begin{split}
&\int_{\Rd}\varphi(x)\,d\rhout^k(x)-\int_{\Rd}\varphi(x)\,d\rhout^h(x)+O(\tau)\\&=-\sum_{j=h}^{k}\tau\left[\int_{\R^d}A(\rhout^{j+1}(x),\rhodt^{j+1}(x))\Delta\varphi(x)\,dx\right.\\&-\int_{\R^d}A_{\rho_1}\left(\rhout^{j+1} (x),\rhodt^{j+1}(x)\right)\rhout^{j+1}(x)\Delta\varphi(x)\,dx\\&\left.+\frac{2}{m_2}\int_{\Rd}(\rhodt^{j+1}(x))^{1-\frac{m_2}{2}} A_{\rho_2}(\rhout^{j+1}(x),\rhodt^{j+1}(x))\nabla\left(\rho_{2,\tau}^{j+1}(x)\right)^{\frac{m_2}{2}}\nabla\varphi(x)\,dx\right]\\& -\sum_{j=h}^{k} \frac{\tau}{2}\int_{\R^{2d}}\nabla H_{1}(x-y)\cdot\left(\nabla\varphi(x)-\nabla\varphi(y)\right)\rhout^{j+1}(y)\rhout^{j+1}(x)\,dy\,dx\\
 & -\sum_{j=h}^{k}\tau\int_{\R^{2d}}\nabla K_{1}(x-y)\cdot\nabla\varphi(x)\rhodt^j(y)\rhout^{j+1}(x)\,dy\,dx,
\end{split}
\end{equation*}
which is equivalent to the following due to the definition of piecewise constant interpolation $\rrhot$:

\begin{equation}\label{finalel}
\begin{split}
&\int_{\Rd}\varphi(x)\,d\rhout(s,x)-\int_{\Rd}\varphi(x)\,d\rhout(t,x)+O(\tau)\\&=-\int_{t}^{s}\int_{\R^d}\left[\vphantom{\nabla\rhodt^{\frac{m_2}{2}}}A(\rhout(\sigma,x),\rhodt(\sigma,x))\Delta\varphi(x) - A_{\rho_1}(\rhout(\sigma,x),\rhodt(\sigma,x))\rhout(\sigma,x)\Delta\varphi(x)\right.\\
& \left.+\frac{2}{m_2}\rhodt^{1-\frac{m_2}{2}}A_{\rho_2}(\rhout(\sigma,x),\rhodt(\sigma,x))\nabla \rhodt^{\frac{m_2}{2}}(\sigma,x)\nabla\varphi(x)\right]\,dx\,d\sigma\\& -\frac{1}{2}\int_{t}^{s}\int_{\R^{2d}}\nabla H_{1}(x-y)\cdot\left(\nabla\varphi(x)-\nabla\varphi(y)\right)\rhout(\sigma,y)\rhout(\sigma,x)\,dy\,dx\,d\sigma\\
 & -\int_{t}^{s}\int_{\R^{2d}}\nabla K_{1}(x-y)\cdot\nabla\varphi(x)\rhodt(\sigma-\tau,y)\rhout(\sigma,x)\,dy\,dx\,d\sigma.
\end{split}
\end{equation}
We are now ready to pass to the limit as $\tau\rightarrow0^+$ in order to recover the first component of a weak measure solution in the sense of the Definition \ref{defweaksolution}. The weak measure convergence in Proposition \ref{propaa} allows us to pass to the limit in the convolution terms in \eqref{finalel}. The diffusion terms involving $A$ and its derivatives can be easily passed to the limit due to corollary \ref{cfi1} combined with assumption (D2) on $A$, where we also use the weak convergence (up to a subsequence) of $\nabla\rhodt^{\frac{m_2}{2}}$ to $\nabla \rho_2^{\frac{m_2}{2}}$ which is a consequence of Theorem \ref{tfi}. Setting $(\rho_1,\rho_2)$ as the $\tau\rightarrow0^+$ limits of $(\rhout,\rhodt)$, a straightforward integration by parts yields
\begin{align*}
& -\int_s^t\int_{\R^d}\left[\vphantom{\nabla\rhodt^{\frac{m_2}{2}}}A(\rho_1(\sigma,x),\rho_2(\sigma,x))\Delta\varphi(x) - A_{\rho_1}(\rho_1(\sigma,x),\rho_2(\sigma,x))\rho_1(\sigma(x))\Delta\varphi(x)\right.\\
& \left.+\frac{2}{m_2}\rho_2^{1-\frac{m_2}{2}}A_{\rho_2}(\rho_1(\sigma,x),\rho_2(\sigma,x))\nabla \rho_2^{\frac{m_2}{2}}(\sigma,x)\nabla\varphi(x)\right]\,dx\,d\sigma\\
& = -\int_s^t\int_{\R^d}\rho_1(\sigma,x)\nabla A_{\rho_1}(\rho_1(\sigma,x),\rho_2(\sigma,x))\nabla \varphi(x)\,dx\,d\sigma.
\end{align*}
Hence, dividing by $s-t$ and taking the limit as $s\searrow t$, one gets the definition of weak solution for the first component. Repeating the same procedure perturbing the second component of $\rrhotnn$ gives the second equation in the Definition \ref{defweaksolution}.
\end{proof}
\appendix
\section{Appendix}
Let us prove the following technical result. 
\begin{lem}\label{convlemma}
Let $f\in H^1(\Rd)$ and $\zeta\in C_c^{\infty}(\Rd;\Rd)$. If we consider a perturbation $P^\ee:=\id+\ee\zeta$ for $\ee>0$, then
\begin{equation}
\frac{f\circ P^\ee-f}{\ee}\underset{\ee\to0}{\longrightarrow}\zeta\cdot\nabla f \quad \text{in}\ L^2(\Rd)
\end{equation}
\end{lem}
\begin{proof}
Recall that for every $f\in L^p(\Rd)$ with $1\le p\le+\infty$ we have
\begin{equation}\label{lpconv}
||f\circ P^\ee-f||_{L^p}\underset{\ee\to0}{\longrightarrow}0
\end{equation}
by means of density argument (see for instance \cite[Lemma 4.3]{B}). Using the fundamental theorem of calculus, the Cauchy-Schwartz inequality and Fubini's Theorem,  we get
\begin{equation*}
\begin{split}
&\int_{\Rd}\left|\frac{f(x+\ee\zeta(x))-f(x)}{\ee}-\zeta(x)\cdot\nabla f(x)\right|^2\,dx\\&=\int_{\Rd}\left|\frac{1}{\ee}\int_0^1\left[\frac{d}{d\tau}f(x+\tau\ee\zeta(x))\right]\,d\tau-\zeta(x)\cdot\nabla f(x)\right|^2\,dx\\&=\int_{\Rd}\left|\int_0^1\zeta(x)\cdot\left[\nabla f(x+\tau\ee\zeta(x))-\nabla f(x)\right]\,d\tau\right|^2\,dx\\&\le||\zeta||_{\infty}^2\int_{\Rd}\int_0^1|\nabla f(x+\tau\ee\zeta(x))-\nabla f(x)|^2\,d\tau\,dx\\&=||\zeta||_{\infty}^2\int_0^1||\nabla f\circ P^{\tau\ee}-\nabla f||_{L^2(\Rd)}^2\,d\tau.
\end{split}
\end{equation*}
Thanks to the \eqref{lpconv}, we have that $||\nabla f\circ P^{\tau\ee}-\nabla f||_{L^2(\Rd)}^2$ is as small as we want, which proves the assertion.
\end{proof}

\section*{Acknowledgments}
The authors acknowledge fruitful discussions with Proff. J. A. Carrillo and F. Santambrogio. AE and SF acknowledge support from the Italian INdAM GNAMPA (National group for Mathematical Analysis, Probability, and their applications) project "Analisi di modelli matematici della fisica, della biologia e delle scienze sociali". The authors are supported by the local fund of the University of L'Aquila "DP-LAND" (Deterministic Particles for Local And Nonlocal Dynamics), and from the Erasmus Mundus programme "MathMods", \url{www.mathmods.eu}.

\end{document}